\documentclass[platex,dvipdfmx]{amsart}
\usepackage{amssymb}
\usepackage{braket}
\usepackage{mathrsfs}
\usepackage{amsrefs}
\usepackage{ifthen}
\usepackage{graphicx}

\usepackage{tikz}
\usetikzlibrary{patterns}
\usepackage{comment} 
\usepackage{mleftright}
\usepackage{hyperref}
\usepackage{cleveref}
\usepackage[all]{xy}
\usepackage{color}
\usepackage{ulem}

\theoremstyle{plain}
\newtheorem{theo}{Theorem}[section]
\crefname{theo}{Theorem}{Theorems}
\Crefname{theo}{Theorem}{Theorems}
\newtheorem{prop}[theo]{Proposition}
\crefname{prop}{Proposition}{Propositions}
\Crefname{prop}{Proposition}{Propositions}
\newtheorem{lem}[theo]{Lemma}
\crefname{lem}{Lemma}{Lemmas}
\Crefname{lem}{Lemma}{Lemmas}
\newtheorem{cor}[theo]{Corollary}
\crefname{cor}{Corollary}{Corollaries}
\Crefname{cor}{Corollary}{Corollaries}
\newtheorem {claim}[theo]{Claim}
\crefname{claim}{Claim}{Claims}
\Crefname{claim}{Claim}{Claims}

\crefname{property}{Property}{Properties}
\Crefname{property}{Property}{Properties}

\crefname{problem}{Problem}{Problems}
\Crefname{problem}{Problem}{Problems}

\crefname{sublem}{Sublemma}{Sublemmas}
\Crefname{sublem}{Sublemma}{Sublemmas}

\theoremstyle{definition}

\crefname{defi}{Definition}{Definitions}
\Crefname{defi}{Definition}{Definitions}

\crefname{notation}{Notation}{Notations}
\Crefname{notation}{Notation}{Notations}

\crefname{convention}{Convention}{Conventions}
\Crefname{convention}{Convention}{Conventions}

\crefname{cond}{Condition}{Conditions}
\Crefname{cond}{Condition}{Conditions}

\crefname{assum}{Assumption}{Assumptions}
\Crefname{assum}{Assumption}{Assumptions}

\theoremstyle{remark}

\crefname{rem}{Remark}{Remarks}
\Crefname{rem}{Remark}{Remarks}

\crefname{ex}{Example}{Examples}
\Crefname{ex}{Example}{Examples}

\crefname{section}{Section}{Sections}
\Crefname{section}{Section}{Sections}
\crefname{subsection}{Subsection}{Subsections}
\Crefname{subsection}{Subsection}{Subsections}
\crefname{figure}{Figure}{Figures}
\Crefname{figure}{Figure}{Figures}

\newtheorem*{acknowledgement}{Acknowledgement}

\newcommand{\Z}{\mathbb{Z}}

\newcommand{\R}{\mathbb{R}}
\newcommand{\C}{\mathbb{C}}

\newcommand{\CP}{\mathbb{CP}}
\newcommand{\scrS}{{\mathscr S}}

\newcommand{\Homeo}{\operatorname{\mathrm{Homeo}}}
\newcommand{\SO}{\operatorname{\mathrm{SO}}}
\newcommand{\Aut}{\operatorname{\mathrm{Aut}}}

\title{A note on exotic families of 4-manifolds}
\author{Tsuyoshi Kato}
\address{Department of Mathematics, Faculty of Science, Kyoto University, Kitashirakawa Oiwake-cho, Sakyo-ku, Kyoto 606-8502, Japan}
\email{tkato@math.kyoto-u.ac.jp}
\author{Hokuto Konno}
\address{Graduate School of Mathematical Sciences, the University of Tokyo, 3-8-1 Komaba, Meguro, Tokyo 153-8914, Japan \\and\\
RIKEN iTHEMS, Wako, Saitama 351-0198, Japan}
\email{konno@ms.u-tokyo.ac.jp}
\author{Nobuhiro Nakamura}
\address{Department of Mathematics, Osaka Medical and Pharmaceutical University, 2-7 Daigakumachi, Takatsuki, Osaka, 569-8686, Japan}
\email{nobuhiro.nakamura.jl@ompu.ac.jp}

\date{\today}

\begin{document}
\maketitle

\begin{abstract}
We present  a pair of smooth fiber bundles over the circle
with a common $4$-dimensional fiber with the following properties: (1) their total spaces are diffeomorphic to each other; (2) they are isomorphic to each other as topological fiber bundles; (3) they are not isomorphic to each other as smooth fiber bundles.
In particular, we exhibit an example with non-simply-connected fiber.
\end{abstract}

\section{Introduction}
It is 
fundamental  in differential topology
to ask  existence and uniqueness of a smooth structure 
on a given topological manifold.
These problems are formulated as the smoothing of a topological manifold and the study of exotic structures respectively.
On fiber bundles, one may ask family versions of these questions, but a subtle phenomenon lurks here:
in \cite{KKN} the authors gave an example of a topological fiber bundle with a $4$-dimensional fiber such that the total space is smoothable as a manifold, but the bundle cannot be smoothed as a fiber bundle (i.e. cannot be reduced to a smooth fiber bundle).
The purpose of this short note is to point out a similar subtlety on the problem of exotic structures for families.

On smooth fiber bundles, one can consider (at least) two types of the notion of `exotic'.
The first one is exotic manifold structures on the total space.
The second is as fiber bundles, namely, if one has smooth fiber bundles $E_{1}$ and $E_{2}$ with a common fiber and base, one may say that $E_{1}$ and $E_{2}$ are `exotic' if they are isomorphic to each other as topological bundles (i.e. bundles whose structure group is the homeomorphism group), but not isomorphic as smooth fiber bundles (i.e. bundles whose structure group is the diffeomorphism group).
We show that, even if $E_{1}$ and $E_{2}$ are exotic in the second sense, it is not necessary that they are exotic in the first sense, in the case that the fiber is of dimension $4$.
First we note such a result for bundles with fiber simply-connected $4$-manifold:

%For smooth fiber bundles $E_{1}$ and $E_{2}$ with common fiber and base, if $E_{1}$ and $E_{2}$ are isomorphic to each other as smooth fiber bundles, of course the total spaces of them are diffeomorphic to each other.
%In this short note, we show that the converse does not hold in general, in the case that the fiber is of dimension $4$:

\begin{prop}
\label{main thm}
There exist smooth fiber bundles $E_{1}$ and $E_{2}$ over $S^{1}$ with a common fiber $X$ which is a smooth closed simply-connected $4$-manifold with the following properties:
\begin{enumerate}
\item the total spaces of $E_{1}, E_{2}$ are diffeomorphic to each other;
\item $E_{1}, E_{2}$ are isomorphic to each other as topological fiber bundles;
\item $E_{1}, E_{2}$ are not isomorphic to each other as smooth fiber bundles.
\end{enumerate}
\end{prop}

\begin{proof}
Let $X=2k\CP^{2}\#l(-\CP^{2})$, where $k \geq 2$ and $l \geq 10k+1$.
This is a closed, oriented, simply-connected, and non-spin $4$-manifold.
Ruberman~\cite{Rub} constructed a self-diffeomorphism $f$ of $X$ such that $f$ is topologically isotopic to the identity of $X$, but not smoothly isotopic to the identity.
Let $E_{1}$ be the product bundle $S^{1} \times X$ over $S^{1}$ and $E_{2}$ be the mapping torus of $f$.
By the properties of $f$, the conditions (2), (3) in the theorem are satisfied.
Therefore it suffices to check (1).

%\textcolor{red}{ここにRubermanの指摘を入れる．}

Since the induced homomorphism of $f$ on $H_2(X;\Z)$ is the identity map, $f$ is pseudo-isotopic to the identity by Kreck's theorem \cite[Theorem 1]{Kreck}.
A pseudo-isotopy between the identity and $f$ is given by a diffeomorphism $F\colon X\times[0,1]\to X\times[0,1]$ such that $F|_{X\times\{0\}}=\mathrm{id}$ and $F|_{X\times\{1\}}=f$.
Then $F$ naturally induces a diffeomorphism between $E_1$ and $E_2$.
(D.~Ruberman kindly pointed out the argument in this paragraph.)
\end{proof}

%\begin{rem}
%As far as $H^{4}(M;\Z)$ is torsion free,
%one can avoid using the condition that $X$ simply-connected for the proof of the coincidence of  $w_{2}, w_{4}, p_{1}$ for $TM_{1}$ and $TM_{2}$.
%To to this, one can, instead, use the facts that the Stiefel--Whitney classes are homotopy invariant (see, for example, page 193 of \cite{May}) and that the rational Pontryagin classes are topological invariant, which is known as the Novikov theorem~\cite{Nov}.
%\end{rem}

%\begin{rem}
%It is classically known that, for an orientable closed smooth manifold of dimension $\leq 3$, the inclusion from the diffeomorphism group into the homeomorphism group is a weak homotopy equivalence.
%Hence a pair of fiber bundles $E_{1}, E_{2}$ satisfying the conditions (2), (3) in \cref{main thm} with such a lower dimensional fiber does not exist.
%\end{rem}

The fact that  the self-diffeomorphism constructed by Ruberman~\cite{Rub} is topologically isotopic to identity relies on the classical result of Quinn~\cite{Quinn1} and Perron~\cite{Perron}.
Their result is proved only for simply-connected $4$-manifolds.
Moreover, Kreck's theorem \cite[Theorem 1]{Kreck} is also proved only for simply-connected 4-manifolds.
Therefore the proof of \cref{main thm} cannot be applied to non-simply-connected $4$-manifolds straightforwardly.
In \cref{label: Non-simply-connected example}, we shall prove a similar statement to \cref{main thm} for a non-simply-connected 4-manifold using a different argument.

\section{Non-simply-connected example}
\label{label: Non-simply-connected example}

\begin{theo}
\label{main thm2}
There exist smooth fiber bundles $E_{1}$ and $E_{2}$ over $S^{1}$ with a common fiber $X$ which is a smooth closed non-simply-connected $4$-manifold with the following properties:
\begin{enumerate}
\item the total spaces of $E_{1}, E_{2}$ are diffeomorphic to each other;
\item $E_{1}, E_{2}$ are isomorphic to each other as topological fiber bundles;
\item $E_{1}, E_{2}$ are not isomorphic to each other as smooth fiber bundles.
\end{enumerate}
\end{theo}

\begin{proof}
Let $X'=2k\CP^{2}\#l(-\CP^{2})$, where $k \geq 2$ and $l \geq 10k+1$, as in the proof of \cref{main thm}.
Let $X = X' \# n(S^{1} \times S^{3})$ for $n>0$.
Let $f$ be the diffeomorphism of $X'$ discussed in the proof of \cref{main thm}, and extend $f$ to a diffeomorphism of $X$ by the identity.
First, we note that  \cref{cor10'}, proved in the next section,
implies that $f$ is topologically isotopic to the identity of $X$. The condition (2) follows from this.

Next we shall prove $f$ is not smoothly isotopic to the identity on $X$. The condition (3) follows from this.
In order to prove this,  we modify Ruberman's argument in  \cite{Rub2} slightly.
First, refine the Seiberg--Witten invariants of diffeomorphisms \cite[Definition2.1, 2.3]{Rub2} to those with $\mu$-map as follows.
Define $\mathrm{SW}(\Gamma,h)$ as a map from $\mathbb{A}(X)$ to  $\Z$ by
\[
\mathrm{SW}(\Gamma,h)(a)=\langle\mu(a),[\tilde{\mathcal{M}}(\Gamma;h)]\rangle,
\]
where $h$ is a $1$-parameter family of perturbations and $\tilde{\mathcal{M}}(\Gamma;h)$ is the $1$-parameter moduli space in \cite[\S2]{Rub2}.
(For $\mu$-map and $\mathbb{A}(X)$, see e.g. \cite{OSThom}.)
The corresponding total invariant $\mathrm{SW}_{tot}(f,\Gamma)$ is similarly defined.
These invariants have the same properties of Ruberman's original ones.
In particular, \cite[Theorem 3.2]{Rub2} holds for the refined one.
The non-triviality of $\mathrm{SW}_{tot}(f,\Gamma)$ for the above diffeomorphism $f$ follows from combining Ruberman's calculation \cite[\S4]{Rub2} with Proposition 2.2 of \cite{OShigher} by Ozsv\'{a}th--Szab\'{o}.
This and \cite[Theorem 3.2]{Rub2} imply that $f$ is  not smoothly isotopic to the identity.
%%%%%%%%%%%%%%%%%%%%%%%%%%

We shall check (1) as follows.
Let $\{f_{t}\}_{t \in [0,1]}$ be a topological isotopy from the identity to $f$.
We have a homeomorphism $F : E_{1} \to E_{2}$ induced from the map $[0,1] \times X \to [0,1] \times X$ defined by $(t,x) \mapsto (t, f_{t}^{-1}(x))$.
%Let $i : X \inc E_{1}$ be the inclusion map onto the fiber over the coset $[0] \in S^{1}$ of $0 \in [0,1]$.
Let $\scrS_{1}$ be the natural smooth manifold structure on $E_{1}$, and $\scrS_{2}$ be the smooth manifold structure on $E_{1}$ defined as the pull-back under $F$ of the natural smooth structure on the total space $E_{2}$.
Define a topological manifold $M$ to be the underlying topological manifold of $E_{1}$, and denote by $M_{j}$ the smooth manifold $(M, \scrS_{j})$ for $j=1,2$.
%Note that $i$ is a smooth map with respect to both of $\scrS_{1}$ and $\scrS_{2}$.
%(However, an inclusion from $X$ onto another fiber of $E_{1}$ is not necessarily a smooth map with respect to $\scrS_{2}$, since $f_{t}^{-1}$ is not necessarily a smooth map for $t \neq 0,1$.)
%\textcolor{red}{ここは単連結でないとき嘘：Since $X$ is a closed simply-connected $4$-manifold, we have that $i : X \inc M$ induces isomorphisms of second and fourth cohomology groups with arbitrary coefficient.}

By a theorem due to \v{C}adek and Van\v{z}ura \cite[Theorem~1 (ii)]{CM},
the isomorphism class of an oriented vector bundle over $M$ of rank $5$ is determined by the characteristic classes $w_{2}, w_{4}, p_{1}$.
Here, to apply \cite[Theorem~1 (ii)]{CM} to $M$, we have used that $H^{4}(M;\Z)$ has no torsion and that $w_{2}(M) \neq 0$, which follows from that $X$ is non-spin.
(See Remark in page 755 of \cite{CM}.)
The coincidence of these characteristic classes $w_{2}, w_{4}, p_{1}$ for $TM_{1}$ and $TM_{2}$ follows from the facts that the Stiefel--Whitney classes are homotopy invariant (see, for example, page 193 of \cite{May}) and that the rational Pontryagin classes are topological invariant, which is known as the Novikov theorem~\cite{Nov}.
%First, note that the normal bundle for the embedding $i : X \inc M_{j}$ is a trivial bundle even for $j=2$, since real line bundles are classified by $w_{1}$ and \textcolor{red}{（ここも単連結でないとき嘘→）we have $H^{1}(X;\Z/2)=0$.}
%Thus we have that $i^{\ast}w_{n}(M_{j}) = w_{n}(X)$, and \textcolor{red}{ここは単連結でないとき嘘：   $i^{\ast}$ are isomorphisms for $n=2,4$ as mentioned above.}
%
%Note that $M_{i}$ can be decomposed into $M_{i} = M_{i}' \#_{f} (S^{1} \times n(S^{1} \times S^{3}))$.
%Here $\#_{f}$ denotes a fiberwise connected sum.
%
%\textcolor{red}{以下の部分は，非単連結の場合，バンドルがねじれているところと，$S^{1} \times S^{3}$の自明バンドルの部分に分ける．バンドルがねじれているところでは以下の議論をそのまま行い，自明バンドルの部分は$T^{2} \times S^{3}$で，これのtangent bundleは自明なので，特性類は消滅．よって結局寄与しない．}
%Hence we have that $w_{n}(M_{1}) = (i^{\ast})^{-1} w_{n}(X) = w_{n}(M_{2})$ for $n=2,4$.
%Similarly we have that $p_{1}(M_{1}) = (i^{\ast})^{-1}p_{1}(X) = p_{1}(M_{2})$.
Therefore it follows from \cite[Theorem~1 (ii)]{CM} that $TM_{1}$ and $TM_{2}$ are isomorphic to each other.
This combined with the following lemma verifies the condition (1) in the assertion of the theorem.
\end{proof}

\begin{lem}
Let $E_{1}$ and $E_{2}$ be smooth manifolds
of the same dimension $5$ or $6$.
If there exists a homeomorphism $F : E_{1} \to E_{2}$ such that $F^{\ast}TE_{2}$ is isomorphic to $TE_{1}$, then $E_{1}$ and $E_{2}$  are mutually diffeomorphic.
\end{lem}

\begin{proof}
Now we make use of some knowledge in high dimensional topology, mainly known as Kirby--Siebenmann theory~\cite{KS}.
As in the proof of \cref{main thm},
let $M$ denote the underlying topological manifold of $E_{1}$, $\scrS_{1}$ denote the smooth manifold structure on $E_{1}$, $\scrS_{2}$ denote the smooth manifold structure on $M$ defined as the pull-back under $F$ of the smooth structure on $E_{2}$, and
denote by $M_{j}$ the smooth manifold $(M, \scrS_{j})$ for $j=1,2$.
Let $\tau M : M \to BTOP$ be the classifying map of the stable tangent microbundle.
The two smooth structures $M_{j}$ on $M$ give rise to lifts, which we denote by $\tau M_{j} : M \to BPL$, of $\tau M$ along the natural map $BPL \to BTOP$.
Since we have seen that $TM_{1}$ and $TM_{2}$ are isomorphic to each other, we have that these lifts $\tau M_{j}$ are homotopic to each other. 
Since $M$ is of high dimension (i.e. $\geq 5$), the homotopy classes of lifts of $\tau M$ are naturally in a bijective correspondence with the concordance classes of PL structures on $M$.
(See, for example, \cite[1.7.2~Corollary]{Rud}.
Hence $M_{1}$ and $M_{2}$ are PL-concordant.
By the fact that ``concordance implies isotopy'' in high dimension for PL structures on a TOP-manifold
(see, for example, page 305 of \cite{Quinn2}), we have that $M_{1}$ and $M_{2}$ are PL isotopic, and in particular PL homeomorphic.
By the fact that $PL/O$ is $6$-connected and $\dim M_{1}=5$,
smooth structures on the PL-manifold $M_{1}$ are unique (see \cite[1.7.8 Remark]{Rud}).
%On the other hand, for the PL manifold $M_{1}$, there is a natural bijection from the set of the concordance classes of smooth structures to $[M_{1},PL/O]$.
%(See, for example,\cite[(4.2)~Theorem]{HM}.)
%Since $\dim M_{1}=5$ and $PL/O$ is $6$-connected (see, for example, page 33 of \cite{MM}), $[M_{1},PL/O]$ consists of a single point.
%From this it follows that the concordance classes of smooth structures on $M_{1}$ regarded as a PL manifold, which is PL homeomorphic to $M_{2}$, are unique, and hence the smooth structure of $M_{1}$ and that of $M_{2}$ are smoothly concordant.
%Now, by the fact that ``concordance implies isotopy'' in high dimension for smooth structures on a PL-manifold
%(see, for example, page 305 of Quinn~\cite{Quinn} again), we have that $M_{1}$ and $M_{2}$ are smoothly isotopic.
Hence $M_{1}$ is diffeomorphic to $M_{2}$.
\end{proof}

\section{Topological mapping class groups}
Throughout this section, $M$ typically denotes a compact smooth/topological $4$-manifold with boundary $S^{3}$,
and $\hat{M}$ is a closed $4$-manifold which is obtained from $M$ by capping the boundary by $D^{4}$.
The following is a relative version of \cite[1.1~Theorem]{Quinn1}, considered for smooth $4$-manifolds.
%The conclusion follows if $M'$ is replaced with the closed manifold $M$ by \cite{Perron} and  \cite{Quinn1}. 

\begin{theo}
\label{relative-quinn}
Let $N$ be a simply-connected oriented closed smooth $4$-manifold.
Let  $M$ be 
$N \sharp (\CP^2\setminus\mathrm{int}\, D^4)$ or $N \sharp (\overline{\CP^2}\setminus\mathrm{int}\, D^4)$.
The natural homomorphism
\[
\pi_0 (\Homeo(M, \partial M))  \to \Aut (H_2(M), \lambda)
\]
is an isomorphism, where $\lambda$ is the intersection form of $M$.
\end{theo}

Before proving \cref{relative-quinn}, we note the following corollary which is used in the proof of \cref{main thm2}.
%
%\begin{cor}
%\label{cor: disk fix}
%Suppose that $M$ is of the form 
%$N \sharp \CP^2$ or $N \sharp \overline{\CP^2}$,
%where $N$ is a simply connected topological $4$-manifold.
%Then the natural map
%\[
%\Gamma^{top}(M') \to \text{ Aut }(H_2(M'), \lambda)
%\]
%is an isomorphism, where $\lambda$ is the intersection form of $M'$.
%\end{cor}
%
%\begin{proof}
%This follows by combining Theorem \ref{relative-quinn} 
%with Proposition \ref{top}.
%\end{proof}

\begin{cor}\label{cor10'}
Let $\hat{M}$ be a topological $4$-manifold of the form 
$N \sharp \CP^2$ or $N \sharp \overline{\CP^2}$ for $N$
as in \cref{relative-quinn}.
Let $f:\hat{M} \to \hat{M}$ be a homeomorphism which is the identity on a $4$-ball in $\hat{M}$ and induces the identity on the intersection form of $\hat{M}$.
For any topological $4$-manifold $L$,
let $\tilde{f}: \hat{M} \sharp L \to \hat{M} \sharp L$ be the connected sum of $f$ with the identity
on $L$, extended along the fixed ball of $f$.
 Then
$\tilde{f}$  is topologically isotopic to the identity.
\end{cor}

The proof of Theorem \ref{relative-quinn} proceeds as follows.
Let $f: (M , \partial{M})
\to (M,\partial{M})$ be a self-homotopy equivalence of a simply-connected topological $4$-manifold with connected boundary.
 Note that if $f$ preserves the orientation, 
 then the restriction $f|\partial M$ on the boundary $\partial M=S^3$
 is homotopic to the identity by the Hopf degree theorem. 
 Therefore we can assume that the restriction $f|\partial M$  is the identity.

  First we recall the following.
  For a homotopy equivalence $f: (M , \partial{M})\to (M, \partial{M})$, its {\it normal invariant} $n(f)$ is defined   (see e.g. \cite[\S9.4]{Ra} or \cite{FQ}), which takes value in   
  \[
  [(M, \partial M) ; (G/Top, *)] \cong H^2(M,\partial M;\Z_2)\oplus H^4(M,\partial M;\Z).
  \]
 \begin{lem}\label{cap}
 Let $M$ be a simply-connected topological $4$-manifold with $\partial M=S^3$.
 If $f: (M , \partial{M})\to (M, \partial{M})$ is a homeomorphism preserving the boundary, then
%the normal invariant is constant 
$n(f)$ is the class of the constant map in $[(M, \partial M) ; (G/Top, *)] $.
\end{lem}

\begin{proof}
Let $\hat{M} = M \cup_{\partial M} D^4$ be the closed 4-manifold obtained from $M$ by capping the boundary by $D^{4}$.
The homeomorphism $f$ extends to a self-homeomorphism
$\hat{f}: \hat{M} \to \hat{M}$ on the closed 4-manifold.
Note that the normal invariant of $\hat{f}$ is constant in
$[\hat{M}, G/Top]$
(see \cite[Proposition~9.48]{Ra}).
Now the conclusion is deduced from the case of $\hat{f}$.
\end{proof}

For a simply-connected compact topological $4$-manifold $M$ with $\partial M=S^3$, let $\widetilde{HE}_{id}(M,\partial M)$ be the set of self-homotopy equivalences  $f: (M , \partial{M}) \to (M, \partial{M})$ which induce the identity on homology and whose restrictions $f : \partial{M} \to \partial{M}$ are self-homeomorphisms.
For $f,g \in \widetilde{HE}_{id}(M,\partial M)$, we define $f \sim g$ if there exists a homotopy between $f$ and $g$ in the space of continuous self-maps of $(M, \partial{M})$ which are self-homeomorphisms on $\partial{M}$.
Evidently $\sim$ is an equivalence relation on $\widetilde{HE}_{id}(M,\partial M)$, and
we define $HE_{id}(M,\partial M) = \widetilde{HE}_{id}(M,\partial M)/\sim$.

%Similarly, one can define a set $HE(M,\partial M)$ just by dropping the condition that self-maps induce the identity on homology.
%Freedman--Quinn~\cite{FQ} defined the {\it normal invariant}
%\[
%n : HE(M,\partial M) \to [(M, \partial M) ; (G/Top, *)].
%\]

\begin{prop}\label{quinn2.1}
Let $M$ be a simply-connected smooth $4$-manifold with $\partial M=S^3$.
Then the normal invariant
\begin{align*}
%n: \ \{ f: (M , \partial{M})
%\to (M,& \partial{M}): \text{ h.e., inducing the identity on homology } 
%\} / \sim  \\
n\colon HE_{id}(M,\partial M)
 \longrightarrow  \ [(M, \partial M) ; (G/Top, *)] 
%\ \subset \ H^{ev}(M, \partial M; \Z)
\end{align*}
is an injection.
\end{prop}

 \begin{proof}
 This is proven in \cref{Proof of Proposition 3.4}.
% The argument is a relative version of the proof of \cite[Proposition 2.1]{Quinn1}.
 \end{proof}
  
  \cref{cap} and \cref{quinn2.1} imply the following corollary.
 
\begin{cor}\label{cor:quinn2.1}
Let $M$ be a simply-connected smooth $4$-manifold with connected boundary.
Then any self-homeomorphism $f$ of $M$ inducing the identity on homology is homotopic to the identity through a homotopy $\{f_{t}\}_{t \in [0,1]}$ which preserves boundary setwise and $f_{t}|\partial M$ are homeomorphisms of $\partial M$ for all $t$.
\end{cor} 
  
Here we need to use some results due to Quinn~\cite{Quinn1}:
  
\begin{prop}[{\cite[2.2 Proposition]{Quinn1}}]
\label{Quinn2.2}
Let $M$ be a simply-connected topological $4$-manifold with connected boundary.
Let $f_{0}, f_{1}$ are self-homeomorphisms of $M$ such that $f_{0}|\partial M=f_{1}|\partial M$.
If $f_{0}, f_{1}$ are homotopic rel $\partial M$, then they are pseudo-isotopic through a pseudo-isotopy which is the identity on the boundary.
\end{prop}

\begin{theo}[{\cite[1.4 Theorem]{Quinn1}}]
\label{Quinn1.4}
Let $M$ be a simply-connected topological $4$-manifold with boundary.
Then a pseudo-isotopy of $M$ which is the identity on the boundary of $M$ is topologically isotopic rel boundary to a topological isotopy.
\end{theo}
 
  As usual, if we say a homotopy/isotopy is rel boundary, it means fixing boundary pointwise for all parameters of the homotopy/isotopy.
 In order to use \cref{Quinn2.2}, we need a homotopy rel boundary.

 \begin{lem}\label{lem:homotopy}
Let $M$ be a simply-connected topological $4$-manifold with boundary of the form $N \sharp (\CP^2\setminus \mathrm{int} D^4)$ or $N \sharp (\overline{\CP^2}\setminus \mathrm{int} D^4)$.
Let $f_0,f_1: (M , \partial{M}) \to (M , \partial{M})$ be homeomorphisms such that the restrictions $f_i|\partial M$ are the identity for $i=0,1$.
 If $f_0$ and $f_1$ are homotopic through a homotopy $\{f^\prime_t\}_{t\in[0,1]}$ preserving boundary setwise such that the restrictions $f^\prime_t|\partial M$ for all $t$ are self-homeomorphisms of $\partial M$, then there exists a homotopy $\{f_t\}_{t \in [0,1]}$ from $f_{0}$ to $f_{1}$ such that $f_t|\partial M$ are the identity for all $t$.
 \end{lem}
\begin{proof}
Let $C=\partial M\times [0,1]$.
Gluing the boundary of $M$ with $\partial M\times\{1\}$ in $C$ by the identity of $\partial M$, we obtain a $4$-manifold $M\cup C$.
We identify $M$ and $M\cup C$ by fixing a homeomorphism between them.
Define the self-homeomorphism $\tilde{f}_0$ of $M\cup C$ by letting $\tilde{f}_0|C$ be the identity of $C$ and $\tilde{f}_0|M=f_0$. 
Note that $\tilde{f}_0$ is isotopic to $f_0$ rel $\partial M$ via the above identification.
Next define the homotopy $\{\tilde{f}_t\}_{t\in[0,1]}$ from $\tilde{f}_0$ to some homeomorphism $\tilde{f}_1$ as follows.
On $M$ in $M\cup C$, let $\tilde{f}_t | M=f_t^\prime$.
For $(x,s)\in\partial M\times[0,1]=C$, let $\tilde{f}_t(x,s)=f^\prime_{st}(x)$.
Note that the restriction of $\tilde{f}_t$ to $\partial(M\cup C) = \partial M\times\{0\}$ is the identity for any $t$, since $\tilde{f}_t(\cdot,0)=f^\prime_{0}(\cdot)$ is the identity map of $\partial M$.

Now we ask whether $\tilde{f}_1$ is isotopic $f_1$  rel $\partial M$ or not.
Since $f_i|\partial M$ are the ideintity for $i=0,1$, the path $\{f_t|\partial M\}_{t\in[0,1]}$  in $\Homeo(\partial M)$ defines an element $\alpha\in \pi_1\Homeo(\partial M)$.
Since $\partial M=S^3$, $\pi_1\Homeo(\partial M)\cong\pi_1\SO(4)\cong\Z_2$.
If $\alpha=0$, then $\tilde{f}_1|C$ is isotopic to the identity of $C$ rel $\partial M\times\{0\}$, and therefore $\tilde{f}_1$ is isotopic to $f_1$ rel $\partial M$.
If $\alpha\neq0$, then $\tilde{f}_1|C$ is isotopic to the Dehn twist.
\begin{claim}[{\cite[Theorem 2.4]{G}}]\label{cl:dehn}
If $M=N \sharp (\CP^2\setminus \mathrm{int} D^4)$ or $N \sharp (\overline{\CP^2}\setminus \mathrm{int} D^4)$, then the Dehn twist $\delta$ on a collar of $\partial M$ is isotopic to the identity rel $\partial M$.
\end{claim}
The proof of the claim is given below.
Assume that $\delta$ is the Dehn twist on $M\cup C$ such that $\delta | M$ is the identity. 
Let $\bar{f}_1$ be the self-homeomorphism of $M\cup C$ which extends $f_1$ on $M$ by the identity over $C$.
Note that $\bar{f}_1$ is isotpic to $f_1$ rel $\partial M$, and $\tilde{f}_1$ is isotopic to $\bar{f}_1\circ\delta$ rel $\partial M$.
By \cref{cl:dehn}, $\bar{f}_1\circ\delta$ is isotopic to $\bar{f}_1$ rel $\partial M$.
 \end{proof}
 
 \begin{proof}[Proof of \cref{cl:dehn}]
 The proof is based on that of \cite[Lemma 3.5]{G}.
 There is an $S^1$-action on $\CP^2$ defined by $\lambda\cdot [x,y,z] = [x,y,\lambda z]$ for $\lambda\in S^1$ and $[x,y,z]\in\CP^2$.
 Then $p_1=[0,0,1]$ and $p_2=[1,0,0]$ are fixed points.
 The $S^1$-action near a fixed point is classified by the representation on the tangent space of the fixed point.
 Let $\C_n$ be the complex $1$-dimensional representation defined by $\lambda\cdot z =\lambda^nz$ for $\lambda\in S^1$ and $z\in\C$.
 Then the representation on $T_{p_1}\CP^2$ is isomorphic to $\C_{-1}\oplus\C_{-1}$, and the representation on $T_{p_2}\CP^2$ is $\C_{0}\oplus\C_{1}$.
Remove two small open $4$-ball which are neighborhoods of $p_1$ and $p_2$ from $\CP^2$.
For each $i=1,2$, let $C_i$  be a collar of the boundary component correspoinding to $p_i$.
Then, by using the $S^1$-action, we can deform the Dehn twist $\delta$ on $C_2$ by isotopy to the two fold composition $(\delta^{-1})^2$ of the inverse Dehn twist $\delta^{-1}$ on $C_1$ which is isotopic to the identity.  
\end{proof}

Now we complete the proof of \cref{relative-quinn}.
 \begin{proof}[Proof of \cref{relative-quinn}]
 By a restult of Boyer \cite{Boyer}, any automorophism in $\Aut (H_2(M), \lambda)$ is realized by a homeomorphism of $M$.
 Therefore the homomorphism in the theorem is surjective.

The injectivity is proved as follows.
Let $f$ be a self-homeomorphism of $M$ inducing the identity on $H_2(M)$  and fixing boundary pointwise.
By \cref{cor:quinn2.1}, $f$ is homotopic to the identity through a homotopy preserving boundary setwise such that the restrictions of the homotopy to $\partial M$ is a family of self-homeomorphisms of $\partial M$, and
by \cref{lem:homotopy}, we can replace the homotopy between $f$ and $\rm{id}$ with a homotopy rel $\partial M$.
Then \cref{Quinn2.2} and \cref{Quinn1.4} imply that $f$ and $\rm{id}$ are isotopic rel $\partial M$.
% \cref{relative-quinn} is verified by combining 
%  Propopsition \ref{quinn2.1} above, 
%  and Proposition $2.2$
%  and Theorem $1.4$ in  \cite{Quinn1}.
\end{proof}

\section{Proof of Proposition \ref{quinn2.1}}
\label{Proof of Proposition 3.4}

We reduce  our case to the proof of Proposition~2.1 in  \cite{Quinn1}.
Let $M$ be a simply connected smooth $4$-manifold with boundary
$\partial M \cong S^3$. 
For $i=0,1$, let
$h^i: (M, \partial M) \to (M, \partial M)$ be  homotopy
equivalences such that the restrictions
$h^i| \partial M$ are homeomorphisms.

Let $\hat{M}: = M \cup D^4$ be the closed $4$-manifold
that is capped off by the disk.
We call the origin
$* =0 \in D^4 \subset \hat{M}$ the base point.
By the boundary condition, we can extend the self-homotopy equivalences $h^{i}$ of $M$ to self-homotopy equivalences of $\hat{M}$ by attaching the cones of $h^i| \partial M$ on $D^{4}$:
\[
h^i: (\hat{M} ,*)\to (\hat{M},*).
\]
Suppose that there is  a homotopy
\[
h_{t}: \hat{M} \to \hat{M}
\]
with $h_{i}\equiv h^i$ for $i=0,1$,
that does not necessarily preserve the base point 
$* \in \hat{M}$.

 \begin{lem}\label{lem4.5}
 Let  $h_{\bullet} :  [0,1] \times (\hat{M}, *) \to (\hat{M},*)$ be a homotopy between 
 homotopy
 equivalences $h_0=h^{0}$ and $h_1=h^{1}$ as above.
 Then there exists a homotopy 
$\{h_{t,s}:  \hat{M} \to \hat{M}\}_{(t,s) \in [0,1]^{2}}$
 such that $h_{0,1}=h^0$, $h_{1,1}=h^1$ and $h_{t,0}=h_t$ and $h_{t,1}^{-1} (*) = \{\ast\}$ hold for all $t \in [0,1]$.
% \textcolor{red}{Instead of  $h_{t,1}^{-1} (*)\subset D^4 \subset \hat{M}$, should we consider $h_{t,1}^{-1} (*)\subset B \subset \hat{M}$?}
 \end{lem}

 \begin{proof}
  Note that
 $h_t^{-1}(*)=\{*\}$
 for $t \in \{ 0, 1\}$. 
 By smooth approximation, we may assume 
that $h_\bullet : [0,1] \times \hat{M} \to \hat{M}$ is smooth and transversal to the base point $\ast$.
Then  we can assume that $(h_{\bullet})^{-1}(*)$
 is a finite disjoint union of copies of circles $S_{j}$ in 
 $[0,1] \times \hat{M}$ with an interval $I$ with $\partial I = \{0,1\} \times \{\ast\}$: 
 \[
 (h_{\bullet})^{-1}(*) =  I \cup 
 S_1 \cup \cdots \cup S_l.
 \]
 
 Since $M$ is simply-connected, there is a homotopy
$\{H_s:
(\hat{M} ,*)\cong 
(\hat{M},*)\}_{s \in [0,1]}$ such that
 $H_0=id$, 
$H_1(S_j) = \{\ast\}$ for all $j=1,\ldots, l$, and $H_{1}(I)=[0,1] \times \{\ast\}$.
Then the composition $h_{t,s}: = H_s \circ h_t: 
   (\hat{M}, *) \to (\hat{M},*)$ satisfies the desired properties.
 \end{proof}
 
% Henceforth, we denote $h_{t,1}$ in \cref{lem4.5} by $h_{t}$, then we may suppose that $h_{t}^{-1}(\ast)=\{\ast\}$ for all $t$.

\begin{prop}\label{global}
 Let  $h_{\bullet} :  [0,1] \times (\hat{M}, *) \to (\hat{M},*)$ be a homotopy between 
 homotopy
 equivalences $h_0=h^{0}$ and $h_1=h^{1}$.
 
Then there exists a homotopy $\{h'_{t,u} : \hat{M} \to \hat{M}\}_{(t,u) \in [0,1]^{2}}$ satisfying the following properties:
\begin{itemize}
\item  $h'_{t,0}=h_t$ for all $t$.
\item  $h'_{i,1}= h^{i}$ for $i=0,1$.
\item $(h'_{t,1})^{-1}(*)  = \{ * \}$ for all $t$.
\item $h'_{t,1}$ give rise to maps between triples
 \[
 h'_{t,1}: (\hat{M},  M, D^4) \to (\hat{M}, M,D^4).
 \]
 \item  The restrictions
 $h'_{t,1}: \partial D^4 \cong \partial D^4$
are homeomorphisms for all $t$.
 \end{itemize}
 \end{prop}
 
 Before giving the proof of \cref{global}, we show  preliminary lemmae.
 Let $rD^4$ be the $4$-ball in $\R^4$ centered at the origin with radius $r$, and $D^4$ the unit $4$-ball.
 Let $R\colon D^4\to D^4$ be the reflection defined by $(x_1,x_2,x_3,x_4)\mapsto (x_1,x_2,x_3,-x_4)$.
  \begin{lem}\label{lem4.4-2}
 Let $F:(D^4, \partial D^4) \to (D^4, \partial D^4) $ be a continuous map of degree $\pm1$.
 If the degree of $F$ is $+1$, then $F$ is homotopic to the identity.
 If the degree of $F$ is $-1$, then $F$ is homotopic to the reflection $R$.
  \end{lem}
 \begin{proof}
% Let $D(r)$ be the $4$-ball in $\R^4$ centered at the origin with radius $r$.
 Define $F^{(r)}\colon rD^4\to rD^4$ by $F^{(r)}(x)=r F(x/r)$.
 If the degree of $F$ is $+1$, then the degree of the restriction $F|\partial D^4$ is also $+1$, and therefore it is homotopic to the identity.
 Let $G_t$ be the homotopy such that $G_0=F|\partial D^4$ and $G_1$ is the identity.
 Then we can construct a homotopy from $F$  to a map $F^\prime$ whose restriction to $\partial D^4$ is the identity as follows.
 For $t\in[0,1]$, assume $D^4=[0,t/2]\times\partial D^4\cup (1-t/2)D^4$, and
 define 
 \[
 F_t\colon[0,t/2]\times\partial D^4\cup (1-t/2)D^4 \to [0,t/2]\times\partial D^4\cup (1-t/2)D^4
 \] 
 by $F_t=F^{(1-t/2)}$ on $(1-t/2)D^4$, and $F_t(s,x)=G_{t/2-s}(x)$ for $(s,x)\in[0,t/2]\times\partial D^4$.
 Then $\{F_t\}$ gives a desired homotopy.
 Next we construct a homotopy from $F^\prime$ to the identity by ``Alexander's trick" (see e.g. \cite[p.17]{KS}).
 Extend $F^\prime$ to $\R^4$ by the identity outside of $D^4$.
 Let $H_t(x) = tF^\prime(x/t)$ for $0<t\leq 1$.
 Then $H_t$ is continuously extended to $H_0$ by the identity, and this gives a homotopy between $F^\prime$ to the identity.
 
 If the degree of $F$ is $-1$, then the degree of $R\circ F$ is $+1$.
By the argument as above, we can find a homotopy $\{I_t\}$ from $R\circ F$ to the identity.
Then $\{R\circ I_t\}$ gives a homotopy from $F$ to $R$.
 \end{proof} 
 
   \begin{lem}\label{last lem}
 Let $F:(1+\epsilon)D^4 \to (1+\epsilon)D^4$ be a continuous map satisfying that
 \begin{itemize}
 \item $F(\partial((1+\epsilon)D^{4})) \subset \partial((1+\epsilon)D^{4})$, $F(D^{4}) \subset D^{4}$, $F(\partial D^{4}) \subset \partial D^{4}$.
  \item The degree of $F|_{(D^{4}, \partial D^{4})} : (D^{4}, \partial D^{4}) \to (D^{4}, \partial D^{4})$ is $\pm1$.
 \end{itemize}
Then we have a homotopy $\{F_{t} : (1+\epsilon)D^4 \to (1+\epsilon)D^4\}_{t \in [0,1]}$ such that
\begin{itemize}
\item $F_{0}=F$.
\item $F_{1}|_{\partial((1+\epsilon)D^{4})} = F|_{\partial((1+\epsilon)D^{4})}$.
\item $F_{1}|_{D^{4}} : D^{4} \to D^{4}$ is a homeomorphism.
\end{itemize}
  \end{lem}
  
  \begin{proof}
  This follows from \cref{lem4.4-2}.
  \end{proof}

\begin{proof}[{Proof of Proposition \ref{global}}]
Fixing a local coordinate near $D^{4}$ in $\hat{M}$,
we regard $D^{4}$ is the unit disk in $\R^{4}$ centered at the origin, and we may suppose that
we have a closed neighborhood $N$ of $D^{4}$ in $\hat{M}$ which is homeomorphic to $D^{4}$, and in the local coordinate, which is given as the disk of radius $2$ centered at the origin, which we denote by
$2D^{4}$.
Henceforth we identify $N$ with $2D^{4}$  via this fixed local coordinate.

 By \cref{lem4.5}, there exists a homotopy $h_{t,s}:  (\hat{M}, *) \to (\hat{M},*)$
 such that $h_{0,1}=h^0$, $h_{1,1}=h^1$ and $h_{t,0}=h_t$ and $(h_{t,1})^{-1} (*) = \{\ast\}$ hold for all $t \in [0,1]$.
 Set $g_{t} := h_{t,1} : \hat{M} \to \hat{M}$.
We claim that there is a homotopy $\{g_{t,\tau} : \hat{M} \to \hat{M}\}_{(t,\tau) \in [0,1]^{2}}$ satisfying the following properties:
\begin{enumerate}
\item  $g_{t,0}=g_t$ for all $t$.
\item  $g_{i,1}= h^{i}$ for $i=0,1$.
\item $g_{t,1}(D^{4}) \subset \mathrm{int}(2D^{4})$ for all $t$.
\item $g_{t,1}(\partial D^{4}) \subset \mathrm{int}(2D^{4}) \setminus \{\ast\}$ for all $t$.
 \end{enumerate}
First, since we have $g_{t}(\ast)=\ast$ for all $t$, there exists $\epsilon>0$ such that
$g_{t}(\epsilon D^{4}) \subset \mathrm{int}(2D^{4})$ for all $t$, where $\epsilon D^{4}$ denotes the disk of radius $\epsilon$ centered at the origin.
Also, since $g_{i}(D^{4})=D^{4}$ for $i=0,1$,
there exists $\epsilon'>0$ such that $g_{t}(D^{4}) \subset \mathrm{int}(2D^{4})$ if $t \in [0,\epsilon'] \cup [1-\epsilon',1]$.
Define a continuous map $\phi : [0,1] \to [\epsilon,1]$ by
\[
\phi(t)
=
\begin{cases}
\frac{\epsilon-1}{\epsilon'}t + 1, & t \in [0,\epsilon'],\\
\epsilon, & t \in [\epsilon', 1-\epsilon'],\\
\frac{1-\epsilon}{\epsilon'}t + \frac{\epsilon+\epsilon'-1}{\epsilon'}, & t \in [1-\epsilon',1].
\end{cases}
\]
Then $\phi(i)=1$ for $i=0,1$.
Moreover, for each $t \in [0,1]$, define a continuous map $\psi_{t} : [1,2] \to [\epsilon,1]$ by 
\[
\psi_{t}(u) = (2-u)\phi(t) + u-1.
\]
Then $\psi_{t}(1)=\phi(t)$, $\psi_{t}(2)=1$, and $\psi_{0}(u)=1$ for all $u$.
Define a homotopy $\{\Phi_{t} : 2D^{4} \to 2D^{4}\}_{t \in [0,1]}$ by 
\[
\Phi_{t}(x)
=
\begin{cases}
	\phi(t)x, & |x| \leq 1,\\
	\psi_{t}(|x|)x, & 1 \leq |x| \leq 2.
\end{cases}
\]
Then $\Phi_{t}(x)=x$ for $|x|=2$ and for all $t$, and $\Phi_{0}=\Phi_{1}=\mathrm{id}_{2D^{4}}$.
Moreover, by the choice of $\epsilon, \epsilon'$ and the definition of $\phi$,
one may check that $\Phi_{t}(D^{4}) \subset 2D^{4}$ for all $t$.
Define a homotopy $\{\Phi_{t,\tau} : 2D^{4} \to 2D^{4}\}_{(t,\tau) \in [0,1]^{2}}$ by
\[
\Phi_{t,\tau}(x) = (1-\tau)x + \tau\Phi_{t}(x).
\]
Then $\Phi_{t,0} = \mathrm{id}_{2D^{4}}$ and $\Phi_{t,1} = \Phi_{t}$ for all $t$, in particular
$\Phi_{0,1}=\Phi_{1,1}= \mathrm{id}_{2D^{4}}$.
Extend each $\Phi_{t,\tau}$ to a  map $\Phi_{t,\tau} : \hat{M} \to \hat{M}$ by the identity on $M$.
%Then $\Phi_{i}(x)=x$ for $i=0,1$ and for all $x \in 2D^{4}$.
Set $g_{t, \tau} := g_{t} \circ \Phi_{t,\tau} : \hat{M} \to \hat{M}$ for $(t,\tau) \in [0,1]^{2}$.
Then it is straightforward to check that $g_{t, \tau}$ is a continuous map and satisfies the desired properties listed above:
For (1), we use $\Phi_{t,0} = \mathrm{id}_{2D^{4}}$.
For (2), we use $\Phi_{0,1}=\Phi_{1,1}= \mathrm{id}_{2D^{4}}$.
For (3), we use $\Phi_{t}(D^{4}) \subset 2D^{4}$.
For (4), we use $\phi(t)x \neq 0$ if $x \neq 0$.

Set $f_{t} := g_{t,1} : \hat{M} \to \hat{M}$ for $t \in [0,1]$.
Then $f_{t}$ satisfies:
\begin{itemize}
\item  $f_{i}= h^{i}$ for $i=0,1$.
\item $f_{t}(D^{4}) \subset \mathrm{int}(2D^{4})$ for all $t$.
\item $f_{t}(\partial D^{4}) \subset \mathrm{int}(2D^{4}) \setminus \{\ast\}$ for all $t$.
 \end{itemize}
 Since $h^{i}$ gives rise to self-maps of $(D^{4}, \partial D^{4})$, $h^{i} : (D^{4}, \partial D^{4}) \to (D^{4}, \partial D^{4})$, and they have degree $\pm1$ since $h^{i} : \partial D^{4} \to \partial D^{4}$ were supposed to be homeomorphisms and extended as the cones to inside $D^{4}$.
Therefore, by the homotopy invariance of the degree,
\[
f_{t} : (D^{4}, \partial D^{4}) \to (2D^{4}, 2D^{4} \setminus \{\ast\}) \simeq (D^{4}, \partial D^{4})
\]
have also degree $\pm1$ for all $t$, and the sign of the degree is independent of $t$.

We can choose $\epsilon''>0$ such that $f_{t}((1+2\epsilon'')D^{4}) \subset 2D^{4}$ for all $t$.
For $(t,s) \in [0,1]^{2}$, we define $F _{t,s} : (1+2\epsilon'')D^{4} \to 2D^{4}$ as follows:
For $0 \leq r \leq 1+\epsilon''$, $x \in \partial D^{4}$, 
\[
F_{(t,s)}(rx) = \left((1-s)+\frac{sr}{\|f_t(rx)\|}\right)f_t(rx).
\]
For $1+\epsilon'' \leq r \leq 1+2\epsilon''$, $x \in \partial D^{4}$, 
\[
F_{(t,s)}(rx) =\left((1-s)+\frac{s}{\|f_t(rx)\|} \left(\frac{\|f_t((1+2\epsilon'')x)\|-(1+\epsilon'')}{\epsilon''}(r-(1+\epsilon''))+(1+\epsilon'')\right)\right)f_t(rx)
\]
This $F_{(t,s)}$ satisfies that
\begin{itemize}
\item $F_{(t,0)}(rx) = f_t(rx)$.
\item $F_{(t,1)}(rx) = r\frac{f_t(rx)}{\|f_t(rx)\|}$ for $0 \leq r \leq 1+\epsilon''$.
\item $F_{(t,s)}((1+2\epsilon'')x) = f_t((1+2\epsilon'')x)$
\end{itemize}
Extend $F_{(t,s)}$ to a self-map $\hat{F}_{(t,s)} : \hat{M} \to \hat{M}$ such that $\hat{F}_{(t,s)}|_{\hat{M}\setminus (1+2\epsilon'')D^4} = f_{t}|_{\hat{M} \setminus (1+2\epsilon'')D^{4}}$.
%
%: for $r \in [1+\epsilon'', 1+2\epsilon'']$, set
%\[
%\hat{F}_{(t,s)}(rx) = \frac{f_{t}((1+2\epsilon'')x)-F_{(t,s)}((1+\epsilon'')x)}{\epsilon''}r-\frac{1+\epsilon''}{\epsilon''}(f_{t}((1+2\epsilon'')x)-F_{(t,s)}((1+\epsilon'')x)),
%\]
%where $x \in \partial D^{4}$.
%This is a homotopy parametrized by $r \in [1+\epsilon'', 1+2\epsilon'']$ connecting $F_{(t,s)}((1+\epsilon'')x)$ with $f_{t}((1+2\epsilon'')x)$, with fixing $(t,s)$, and thus we have a self-map $\hat{F}_{(t,s)} : \hat{M} \to \hat{M}$ such that
%\begin{itemize}
%\item $\hat{F}_{(t,s)}|_{(1+\epsilon'')D^{4}} = F_{(t,s)}$ for all $(t,s) \in [0,1]^{2}$. 
%\item $\hat{F}_{(t,s)}|_{\hat{M} \setminus (1+2\epsilon'')D^{4}} = f_{t}|_{\hat{M} \setminus (1+2\epsilon'')D^{4}}$  for all $(t,s) \in [0,1]^{2}$. 
%\end{itemize}
%Then we have
%\begin{itemize}
%\item $\hat{F}_{(t,0)}(rx) = f_t(rx)$,
%\item $\hat{F}_{(t,1)}(rx) =  r\frac{f_t(rx)}{\|f_t(rx)\|}$,
%\end{itemize}
%where $0 \leq r \leq 1+\epsilon''$, $x \in \partial D^{4}$.
Notice that $\hat{F}_{(t,1)}$ gives rise to a map $\hat{F}_{(t,1)} : (1+\epsilon)D^{4} \to (1+\epsilon'')D^{4}$ with the property that $\hat{F}_{(t,1)}(rD^{4}) \subset rD^{4}$ for all $r \in [0,1+\epsilon'']$.
Applying Lemma~\ref{last lem} to $\hat{F}_{(t,1)}$, 
%one can deform $h_{t,1}$ by homotopy near the boundary 
%such that the restrictions 
%$f_{t}:  \partial D^4 \to  \partial D^4$
%are homeomorphisms for all $t$, with the condition that $f_{t}^{-1}(*)=*$.
and if necessary, deforming $\hat{F}_{(t,1)}$ by homotopy so that
$\hat{F}_{(t,1)}(M) \subset M$ holds,
concatenating the all four homotopies above (the homotopy from \cref{lem4.5}, $g_{t,\tau}$, $\hat{F}_{{t,s}}$, the homotopy from Lemma~\ref{last lem}),
we obtain the desired homotopy $\{h_{t,s}' :  \hat{M} \to \hat{M}\}_{(t,s) \in [0,1]^{2}}$
from $h_{t}$ to $f_{t}$.
\end{proof}

In particular we have an injection
\[
% \{(M, \partial M) \} / \sim \ \to \  \{ (\hat{M},*) \}/ \sim.
HE_{id}(M,\partial M) \to HE_{id}(\hat{M},*).
 \]
 
Here $HE_{id}(M,\partial M)$ was defined before \cref{quinn2.1},
and  $HE_{id}(\hat{M},*)$ is defined similarly.
Precisely, define $\widetilde{HE}_{id}(\hat{M},*)$ as the set of self-homotopy equivalences  $f: (\hat{M} , \ast) \to (\hat{M}, \ast)$ which induce the identity on homology.
For $f,g \in \widetilde{HE}_{id}(\hat{M},*)$, we define $f \sim g$ if there exists a based homotopy between $f$ and $g$ in the space of continuous self-maps of $(M, \ast)$.
We define $HE_{id}(\hat{M},*) = \widetilde{HE}_{id}(\hat{M},*)/\sim$.
Just by dropping the conditions on the base point,
we can similarly define $HE_{id}(\hat{M})$.
 
 \begin{cor}
 \label{cor last}
 The composition
 \[ n \colon
% \{(M, \partial M) \} / \sim  \ \to \  \{ (\hat{M},*) \}/ \sim
HE_{id}(M,\partial M) \to HE_{id}(\hat{M},*)
 \ \to  \
 [(\hat{M}, \ast) ; (G/Top, *)] 
 \]
 is an injection.
 \end{cor}

 \begin{proof}
The normal map
%$n : \{ \hat{M} \}/ \sim  \ \to \
$n \colon HE_{id} (\hat{M})  \ \to \
[\hat{M},G/Top] $ can be identified with the based version
\[
%n :  \{ (\hat{M}, *)\}/ \sim \ \to \
n \colon  HE_{id}(\hat{M}, *)\ \to \
[(\hat{M},\ast),(G/Top,*)]
\]
since $\pi_{1}(\hat{M})=1$ and $\pi_1(G/ Top)=1$.
It follows from \cite[2.1~Proposition]{Quinn1} that $n \colon HE_{id} (\hat{M}) \ \to \
[\hat{M},G/Top] $ is an injection.
The assertion follows from this and the injectivity of $HE_{id}(M,\partial M) \to HE_{id}(\hat{M},*)$.
\end{proof}

\begin{proof}[Proof of \cref{quinn2.1}]
The normal invariant \[
n\colon HE_{id}(M,\partial M)
 \longrightarrow  \ [(M, \partial M) ; (G/Top, *)] 
 \]
 factors through 
  \[ n \colon
% \{(M, \partial M) \} / \sim  \ \to \  \{ (\hat{M},*) \}/ \sim
HE_{id}(M,\partial M) \to  \
 [(\hat{M}, \ast) ; (G/Top, *)] 
 \]
 via the map $[(M, \partial M) ; (G/Top, *)]  \to  [(\hat{M}, \ast) ; (G/Top, *)] $.
 Thus the assertion of \cref{quinn2.1} follows from \cref{cor last}.
\end{proof}

\begin{acknowledgement}
The authors would like to thank Daniel Ruberman for pointing out Kreck's theorem can be used in the proof of Proposition~1.1.
Tsuyoshi Kato was supported by JSPS Grant-in-Aid for Scientific Research (B) No.17H02841 and 
JSPS Grant-in-Aid for Scientific Research on Innovative Areas (Research in a proposed research area) No.17H06461.
Hokuto Konno was partially supported by JSPS KAKENHI Grant Numbers 16J05569, 17H06461, 19K23412, 21K13785.
Nobuhiro Nakamura was supported by JSPS Grant-in-Aid for Scientific Research (C) No.19K03506.
\end{acknowledgement}

\end{document}